\documentclass[a4paper,12pt,reqno]{amsart}
\usepackage{lmodern,amsmath,amssymb,mathtools,amsthm}
\usepackage{enumerate}
\usepackage{graphics}
\usepackage{graphicx}
\usepackage{cite}
\usepackage{bigints}
\usepackage{epstopdf}
\usepackage{floatrow}
\usepackage{caption}
\usepackage{epsfig}
\usepackage{subcaption}
\usepackage{cleveref}
\usepackage{morefloats}
\usepackage{geometry}% just to show the page margins
\usepackage{graphicx}% remove demo option in actual document

\newtheorem{lemma}{Lemma}
\newtheorem{corollary}{Corollary}
\newtheorem{definition}{Definition}
\newtheorem{theorem}{Theorem}
\newtheorem{remark}{Remark}
\newtheorem{example}{Example}
\newtheorem*{note}{Note}

\newtheorem*{problem}{Problem}
\newtheorem*{conjecture}{Conjecture}
\numberwithin{example}{section}
\numberwithin{equation}{section}
\numberwithin{theorem}{section}
\numberwithin{lemma}{section}
\numberwithin{definition}{section}
\numberwithin{corollary}{section}

\usepackage{blindtext}
\title{Sufficient Conditions and Radius Problems for a starlike Class Involving a Differential Inequality}

\author{
Lateef Ahmad Wani
}
\address{
Department of  Mathematics  \\
Indian Institute of Technology, Roorkee-247 667,
Uttarkhand,  India
}
\email{lateef17304@gmail.com}

\author{
A. Swaminathan
}
\address{
Department of  Mathematics  \\
Indian Institute of Technology, Roorkee-247 667,
Uttarkhand,  India
}
\email{swamifma@iitr.ac.in, mathswami@gmail.com}

\allowdisplaybreaks
\bigskip

\begin{document}
\pagestyle{myheadings}
\begin{abstract}
Let $\mathcal{A}_n$ be the class of analytic functions $f(z)$ of the form $f(z)=z+\sum_{k=n+1}^\infty a_kz^k,n\in\mathbb{N}$ and let
\begin{align*}
\Omega_n:=\left\{f\in\mathcal{A}_n:\left|zf'(z)-f(z)\right|<\frac{1}{2},\; z\in\mathbb{D}\right\}.
\end{align*}
We make use of differential subordination technique to obtain sufficient conditions for the class $\Omega_n$, and then employ these conditions to construct functions which involve double integrals and members of $\Omega_n$. We also consider a subclass $\widehat{\Omega}_n\subset\Omega_n$ and obtain subordination results for members of $\widehat{\Omega}_n$ besides a necessary and sufficient condition. Writing $\Omega_1=\Omega$, we obtain inclusion properties of $\Omega$ with respect to functions defined on certain parabolic regions
and as a consequence, establish a relation connecting the parabolic starlike class $\mathcal{S}_p$ and the uniformly starlike $UST$. Various radius problems for the class $\Omega$ are considered and the sharpness of the radii estimates is obtained analytically besides graphical illustrations.
\end{abstract}


\maketitle

\pagestyle{myheadings}
\markboth{
Lateef Ahmad Wani and A. Swaminathan
}{
Sufficient Conditions and Radius Problems for a starlike Class ...
}
\section{Introduction}
Let $\mathbb{C}$ be the set of complex numbers and let $\mathcal{H}=\mathcal{H}(\mathbb{D})$ be the totality of functions $f(z)$ that are analytic in the open unit disc
$\mathbb{D}:=\{z\in\mathbb{C}:|z|<1\}$. For $a\in\mathbb{C}$ and $n\in\mathbb{N}:=\{1,2,3,...\}$, we define the function classes $\mathcal{H}_n(a)\text{ and } \mathcal{A}_n$ as follows:
\begin{align*}
\mathcal{H}_n(a):=\left\{f\in\mathcal{H}:f(z)=a+\sum_{k=n}^\infty a_kz^k,\;a_k\in\mathbb{C}\right\}
\end{align*}
and
\begin{align*}
\mathcal{A}_n:=\left\{f\in\mathcal{H}:f(z)=z+\sum_{k=n+1}^\infty a_kz^k,\;a_k\in\mathbb{C}\right\}.
\end{align*}
In particular, we write $\mathcal{A}:=\mathcal{A}_1$. For $0\leq\alpha<1$, let $\mathcal{S}^*(\alpha)$ and $\mathcal{C}(\alpha)$ be the subclasses of $\mathcal{A}$ which consist of functions that are, respectively, starlike and convex of order $\alpha$. Analytically,
\begin{align*}
\mathcal{S}^*(\alpha):=\left\{f:\mathrm{Re}\left(\frac{zf'(z)}{f(z)}\right)>\alpha\right\}
\text{ and }
\mathcal{C}(\alpha):=\left\{f:\mathrm{Re}\left(1+\frac{zf''(z)}{f'(z)}\right)>\alpha\right\}.
\end{align*}
Further, $\mathcal{S}^*:=\mathcal{S}^*(0)$ and $\mathcal{C}:=\mathcal{C}(0)$ are the well-known classes of starlike and convex functions in $\mathbb{D}$. We note that $f\in\mathcal{C}$ if and only if $zf'\in\mathcal{S}^*$ and  $\mathcal{C}\subsetneq\mathcal{S}^*\subsetneq\mathcal{S}$, where $\mathcal{S}$ is the collection of all functions $f\in\mathcal{A}$ that are univalent in $\mathbb{D}$. For further details related to these classes, we refer to the monograph of Duren \cite{Duren-1983-Book-UFs}. In our discussion, we make use of the following concepts from the literature.
\begin{definition}[Subordination \cite{Little-Wood-1925-SUB}]
Let $f,g\in\mathcal{H}$. We say that $f$ is subordinate to $g$, written as $f\prec g$, if there exists a
function $w$, analytic in $\mathbb{D}$ with $w(0)=0$ and $|w(z)|<1$, such that $f(z)=g(w(z))$.
\end{definition}
Moreover, if the function $g(z)$ is univalent in $\mathbb{D}$, then $f\prec g$ if and only if
\begin{align*}
f(0)=g(0) \mbox{ and } f(\mathbb{D})\subset g(\mathbb{D}).
\end{align*}
\begin{definition}[{\tt Hadamard product}]
Let $f(z)=z+\sum_{k=2}^{\infty}a_kz^k\in\mathcal{A}$ and $g(z)=z+\sum_{k=2}^{\infty}b_kz^k\in\mathcal{A}$, then the Hadamard product (convolution) of $f$ and $g$ is denoted by $f*g$ and is defined as the analytic function
\begin{align*}
h(z)=(f*g)(z)=z+\sum_{k=2}^{\infty}a_kb_kz^k,\quad z\in\mathbb{D}.
\end{align*}
\end{definition}
Under the operation of Hadamard product, the function $\ell(z)=z/(1-z)=z+\sum_{k=2}^\infty z^k$, which  maps $\mathbb{D}$ onto the right-half plane $\mathrm{Re}(w)>-1/2$, plays the role of identity element. That is, for any function $f\in\mathcal{A}$,
\begin{align*}
(f*\ell)(z)=f(z)=(\ell*f)(z).
\end{align*}
\begin{definition}[{\tt Subordinating Factor Sequence} \cite{HS-Wilf}]
A sequence $\{s_k\}_{k=1}^\infty$ of complex numbers is said to be a subordinating factor sequence if for every convex function
\begin{align*}
\nu(z)=z+\sum_{k=2}^\infty c_kz^k, \quad z\in\mathbb{D}
\end{align*}
we have the subordination given by
\begin{align*}
\sum_{k=1}^\infty s_kc_kz^k\prec \nu(z), \quad (c_1=1).
\end{align*}
\end{definition}
Let $\psi:\mathbb{C}^2\times\mathbb{D}\to\mathbb{C}$ be a complex function, and let $h:\mathbb{D}\to\mathbb{C}$ be univalent. If $p\in\mathcal{H}$ satisfies the first-order differential subordination
\begin{align}\label{Psi}
\psi(p,zp';z)\prec{h(z)},\quad z\in\mathbb{D},
\end{align}
then $p$ is called a solution of the differential subordination. If $q$ is univalent and $p{\prec}q$ for all $p$ satisfying (\ref{Psi}), then $q(z)$ is said to be a dominant of (\ref{Psi}). A dominant \^{q} that satisfies \^{q}$\prec q$ for all dominants $q$ of (\ref{Psi}) is called the best dominant of (\ref{Psi}). The best dominant is unique up to the rotations of $\mathbb{D}$. For further details related to results on differential subordinations, we refer to the monograph of Miller and Mocanu \cite{Miller-Mocanu-Book-2000-Diff-Sub} (see also \cite{Bulboaca-2005-Diff-Sub-Book}).
\begin{lemma}[{\cite[Theorem 3.1b, p. 71]{Miller-Mocanu-Book-2000-Diff-Sub}}]\label{Lemma1-Subordination-Prel-MM-Book-DIP}
Let $h(z)$ be a convex function in $\mathbb{D}$ with $h(0)=a,\,\gamma\neq0$ and $\mathrm{Re}\gamma\geq0$. If
$p(z)\in\mathcal{H}_n(a)$ and
\begin{align*}
p(z)+\gamma^{-1}zp'(z)\prec h(z),
\end{align*}
then
\begin{align*}
p(z)\prec q(z)\prec h(z),
\end{align*}
where
\begin{align*}
q(z)=\frac{\gamma}{nz^{\frac{\gamma}{n}}}\int_0^zh(\xi)\xi^{\frac{\gamma}{n}-1}d\xi.
\end{align*}
The function $q(z)$ is convex and is the best dominant.
\end{lemma}
Given $n\in\mathbb{N}$, we define a new function class $\Omega_n$ as follows
\begin{align*}
\Omega_n:=\left\{f\in\mathcal{A}_n:\left|zf'(z)-f(z)\right|<\frac{1}{2},\; z\in\mathbb{D}\right\}.
\end{align*}

For $n=1$, the class $\Omega_1:=\Omega$ was recently introduced and studied by Peng and Zhong \cite{2017-Acta-Math-Omega-Peng-Zhong}. The authors in \cite{2017-Acta-Math-Omega-Peng-Zhong} have shown that $\Omega$ is a subset of $\mathcal{S}^*$, and hence established that the members of $\Omega$ are univalent in $\mathbb{D}$. Besides discussing several geometric properties of the members of $\Omega$, they \cite{2017-Acta-Math-Omega-Peng-Zhong} proved that the radius of convexity for $\Omega$ is $\frac{1}{2}$, and that $\Omega$ is closed under the Hadamard product, i.e., if $f_1,f_2\in\Omega$, then $f_1*f_2\in\Omega$. The closedness under Hadamard product makes the class $\Omega$ more important, as this is not true, in general, for several subclasses of $\mathcal{S}$ (e.g., $\mathcal{S}^*$). Recently, Obradovi\'{c} and Peng \cite{2018-BMMS-Omega-Obra-Peng} considered the class $\Omega$ and gave two sufficient conditions for functions $f\in\mathcal{A}$ to belong to the class $\Omega$.

In this paper, we consider the class $\Omega_n$, which is in some sense, a natural generalization of $\Omega$. The paper is organized as follows: In \Cref{Section-Sufficient-Conditions-FMP}, we use differential subordination to obtain sufficient conditions for the functions $f\in\mathcal{A}_n$ to be in the class $\Omega_n$. Moreover, we use these results to construct functions of the form
\begin{align*}
f(z)=\iint_0^1\mathcal{J}(s,t,z)dsdt,
\end{align*}
and obtain conditions on the kernel function $\mathcal{J}$ so that $f\in\Omega_n$. 
%%The previous sufficient conditions derived by Obradovi\'{c} and Peng \cite{2018-BMMS-Omega-Obra-Peng} are obtained as special cases.
In \Cref{Section-Subclass-OmegaCap-FMP}, we consider a subclass $\widehat{\Omega}_n\subset\Omega_n$, for which the sufficient conditions obtained in \Cref{Section-Sufficient-Conditions-FMP} become necessary also, and prove a subordination result for the elements of $\widehat\Omega_n$. In \Cref{Section-Inclusion-Properties-FMP} and \Cref{Section-Radii-Problems-FMP}, we restrict ourselves to the class $\Omega$ (i.e., we fix $n=1$). In \Cref{Section-Inclusion-Properties-FMP}, inclusion relations between $\Omega$, the parabolic starlike class $\mathcal{S}_p$ and the uniformly starlike class $UST$ are studied, and as a consequence a remarkable result connecting $\mathcal{S}_p$ and $UST$ is derived. In \Cref{Section-Radii-Problems-FMP}, several newly constructed starlike classes are introduced and the corresponding radius problems for the class $\Omega$ are settled. Also, sharpness of the radii estimates is illustrated graphically. Interesting problems for future work are proposed in \Cref{Section-Open-Problems-FMP}.
\section{Sufficient Conditions for the class $\Omega_n$}\label{Section-Sufficient-Conditions-FMP}
In this section, we consider some conditions on the functions $f\in\mathcal{A}_n$, so that they belong to the class $\Omega_n$.
\begin{theorem}\label{Theorem2-Suff-Cnd-Omega-FMP}
Let $n\in\mathbb{N}$ and $\gamma\geq1$. If $f\in\mathcal{A}_n$ satisfies
\begin{align}\label{Eq:1-Th2-Suff-Cnd-FMP}
\left|zf''(z)+(\gamma-1)\left(f'(z)-\frac{f(z)}{z}\right)\right|<\frac{n+\gamma}{2},
\end{align}
then $f\in\Omega_n$. The result is sharp for the function
\begin{align*}
\widehat{f}_{n,\mu}(z)=z+\frac{\mu}{2n}z^{n+1},\quad |\mu|=1.
\end{align*}
\end{theorem}
\begin{proof}
We rewrite the inequality (\ref{Eq:1-Th2-Suff-Cnd-FMP}) in terms of subordination as
\begin{align}\label{Eq:2-Th2-Suff-Cnd-FMP}
zf''(z)+(\gamma-1)\left(f'-\frac{f(z)}{z}\right)\prec\frac{n+\gamma}{2}z.
\end{align}
Setting
\begin{align*}
p(z)=f'(z)-\frac{f(z)}{z}=\displaystyle\sum_{k=n}^\infty(ka_{k+1})z^k\in\mathcal{H}_n(0),
\end{align*}
the subordination (\ref{Eq:2-Th2-Suff-Cnd-FMP}) takes the form
\begin{align*}
\gamma p(z)+zp'(z)\prec\frac{n+\gamma}{2}z:=h(z).
\end{align*}
It can be easily seen that $h(z)$ is convex in $\mathbb{D}$ and $h(0)=0=p(0)$. Hence it follows from \Cref{Lemma1-Subordination-Prel-MM-Book-DIP} that
\begin{align*}
p(z)\prec\frac{1}{nz^{\frac{\gamma}{n}}}\int_0^z\left(\frac{n+\gamma}{2}\xi\right)\xi^{\frac{\gamma}{n}-1}d\xi=\frac{1}{2}z.
\end{align*}
This further implies that
\begin{align}\label{Eq:1-Th1-Suff-Cnd-FMP}
\left|f'(z)-\frac{f(z)}{z}\right|<\frac{1}{2}.
\end{align}
Now making use of (\ref{Eq:1-Th1-Suff-Cnd-FMP}), and the fact that $f(0)=0$, it follows that
\begin{align*}
\left|zf'(z)-f(z)\right|=|z|\left|f'(z)-\frac{f(z)}{z}\right|<\frac{1}{2}.
\end{align*}
This proves that $f\in\Omega_n$. For the function $\widehat{f}_{n,\mu}(z)=z+\frac{\mu}{2n}z^{n+1}$ with $|\mu|=1$, we have
\begin{align*}
\left|z\widehat{f}_{n,\mu}''(z)+(\gamma-1)\left(\widehat{f}_{n,\mu}'(z)-\frac{\widehat{f}_{n,\mu}(z)}{z}\right)\right|
                            =\left|\frac{n+\gamma}{2}\mu{z}^n\right|<\frac{n+\gamma}{2}.
\end{align*}
That is, $\widehat{f}_{n,\mu}(z)$ satisfies the condition of \Cref{Theorem2-Suff-Cnd-Omega-FMP} and hence belongs to $\Omega_n$. Indeed, for $z\in\mathbb{D}$, we have
\begin{align*}
\left|z\widehat{f}_{n,\mu}'(z)-\widehat{f}_{n,\mu}(z)\right|
                      =\left|\left(z+\mu\frac{(n+1)}{2n}z^{n+1}\right)-\left(z+\frac{\mu}{2n}z^{n+1}\right)\right|
                          =\left|\frac{1}{2}z^{n+1}\right|
                          <\frac{1}{2}.
\tag*{\qedhere}
\end{align*}
\end{proof}
Letting $\gamma=1$ in \Cref{Theorem2-Suff-Cnd-Omega-FMP} yields the following result.
\begin{corollary}\label{CoroTheorem1-Suff-Cnd-Omega-FMP}
If $f\in\mathcal{A}_n$ satisfies
\begin{align*}
\left|zf''(z)\right|<\frac{n+1}{2},
\end{align*}
then $f\in\Omega_n$. The result is sharp.
\end{corollary}
Fixing $n=1$, and then taking $\gamma=1$ and $\gamma=2$ in \Cref{Theorem2-Suff-Cnd-Omega-FMP}, respectively, we obtain the following sufficient conditions proved by Obradovi\'{c} and Peng \cite{2018-BMMS-Omega-Obra-Peng}
\begin{corollary}[{\cite[Theorem 2]{2018-BMMS-Omega-Obra-Peng}}]
If $f\in\mathcal{A}$ satisfies $\left|zf''(z)\right|<1$, then $f\in\Omega$. The number $1$ is best possible.
\end{corollary}
\begin{corollary}[{\cite[Theorem 3]{2018-BMMS-Omega-Obra-Peng}}]
Let $f\in\mathcal{A}$. If
\begin{align*}
\left|z^2f''(z)+zf'(z)-f(z)\right|<\frac{3}{2},
\end{align*}
then $f\in\Omega$. The number $\frac{3}{2}$ is best possible.
\end{corollary}
\begin{note}
For brevity, we fix $\widehat{f}_{n,\mu}(z)=z+\frac{\mu}{2n}z^{n+1}\;(|\mu|=1)$ and write $\widehat{f}_{1}(z)=\widehat{f}_{1,1}(z)=z+\frac{1}{2}z^{2}$.
\end{note}
\begin{figure}[h]
\includegraphics{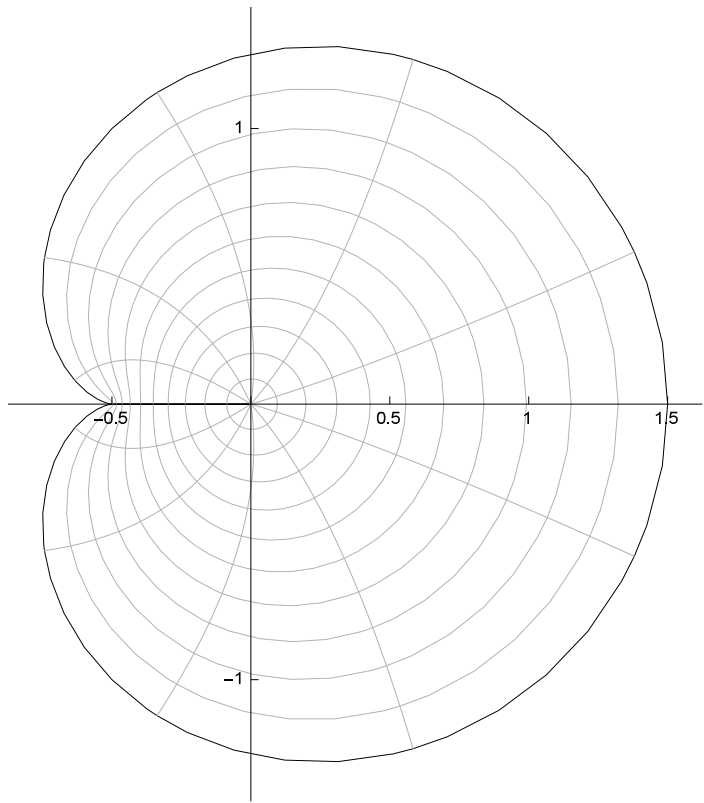}
\caption{$\widehat{f}_1(\mathbb{D})$, where $\widehat{f}_1(z)=z+z^2/2$.}
\label{Figure-f1-Cap-FMP}
\end{figure}
\begin{theorem}\label{Theorem4-Suff-|an|-Omega-FMP}
Let $\gamma\geq1,\,n\in\mathbb{N}$, and let
\begin{align*}
f(z)=z+\sum_{k=n+1}^{\infty}a_kz^k\in\mathcal{A}_n, \quad z\in\mathbb{D}.
\end{align*}
If
\begin{align}\label{Eq:1-Theorem4-Suff-|an|-Omega-FMP}
\sum_{k=n+1}^{\infty}\left(k-1\right)\left(k+\gamma-1\right)|a_k|\leq\frac{n+\gamma}{2},
\end{align}
then $f\in\Omega_n$. Equality holds for the function $\widehat{f}_{n,\mu}(z)$.
\end{theorem}
\begin{proof}
Suppose that (\ref{Eq:1-Theorem4-Suff-|an|-Omega-FMP}) holds, then for $z\in\mathbb{D}$, we have
\begin{align*}
\left|zf''(z)+(\gamma-1)\left(f'-\frac{f(z)}{z}\right)\right|
                  &=\left|\sum_{k=n+1}^{\infty}k(k-1)a_kz^{k-1}+(\gamma-1)\left(\sum_{k=n+1}^{\infty}(k-1)a_kz^{k-1}\right)\right|\\
                   &=\left|\sum_{k=n+1}^{\infty}(k-1)[k+\gamma-1]a_kz^{k-1}\right|\\
                    %&\leq\sum_{k=n+1}^{\infty}(k-1)[k+\gamma-1]|a_k||z|^{k-1}\\
                     &<\sum_{k=n+1}^{\infty}(k-1)[k+\gamma-1]|a_k|\\
                      &\leq\frac{n+\gamma}{2}.
\end{align*}
Now, $f(z)$ is in the class $\Omega_n$ follows from \Cref{Theorem2-Suff-Cnd-Omega-FMP}.\\
It is clear that for $\widehat{f}_{n,\mu}(z)$, we have $a_{n+1}=\frac{\mu}{2n}$ and $a_{k}=0$ for all $k\geq n+2$, so that
\begin{align*}
\sum_{k=n+1}^{\infty}(k-1)[k+\gamma-1]|a_k|=(n+1-1)[n+1+\gamma-1]\frac{1}{2n}=\frac{n+\gamma}{2}.
\tag*{\qedhere}
\end{align*}
\end{proof}
For $\gamma=1$, \Cref{Theorem4-Suff-|an|-Omega-FMP} gives
\begin{corollary}\label{Corollary1-Theorem4-Omega-FMP}
If $f(z)=z+\sum_{k=n+1}^{\infty}a_kz^k\in\mathcal{A}_n$ satisfies
\begin{align*}
\sum_{k=n+1}^{\infty}k(k-1)|a_k|\leq\frac{n+1}{2},
\end{align*}
then $f\in\Omega_n$. The result is sharp.
\end{corollary}
If we fix $n=1$, and then take $\gamma=1$ and $\gamma=2$ in \Cref{Theorem4-Suff-|an|-Omega-FMP}, respectively, we obtain the following sufficient conditions for the function class $\Omega$.
\begin{corollary}\label{Corollary2-Theorem4-Omega-FMP}
Let $f(z)=z+\sum_{k=2}^{\infty}a_kz^k\in\mathcal{A}$. If
\begin{align*}
\sum_{k=2}^{\infty}k(k-1)|a_k|\leq1,
\end{align*}
then $f\in\Omega$ and the result is sharp.
\end{corollary}
\begin{corollary}\label{Corollary3-Theorem4-Omega-FMP}
Let $f(z)=z+\sum_{k=2}^{\infty}a_kz^k\in\mathcal{A}$ satisfies
\begin{align*}
\sum_{k=2}^{\infty}(k^2-1)|a_k|\leq\frac{3}{2}.
\end{align*}
Then $f\in\Omega$ and the result is sharp.
\end{corollary}
\begin{remark}
In both the results, \Cref{Corollary2-Theorem4-Omega-FMP} and \Cref{Corollary3-Theorem4-Omega-FMP}, the equality holds for the function $\widehat{f}_{1,\mu}(z)=z+\frac{\mu}{2}z^2$ with $|\mu|=1$.
\end{remark}
\begin{example}
Consider the function
\begin{align*}
\vartheta_n(z)=z-\frac{3}{20(n+1)}z^{n+1}+\frac{6}{25(n+2)}z^{n+2}+\frac{1}{10(n+3)}z^{n+3}.
\end{align*}
We have,
\begin{align*}
\sum_{k=n+1}^{\infty}k(k-1)|a_k|&=n(n+1)\left|-\frac{3}{20(n+1)}\right|+(n+1)(n+2)\left|\frac{6}{25(n+2)}\right|\\
                      %&\hspace{6em} +(n+2)(n+3)\left|\frac{1}{10(n+3)}\right|\\
                      &\hphantom{(n+1)\left|-\frac{3}{20(n+1)}\right|+(n+1)}+(n+2)(n+3)\left|\frac{1}{10(n+3)}\right|\\
     &=\frac{3n}{20}+\frac{6(n+1)}{25}+\frac{(n+2)}{10}\\
          %%%%&=\frac{3(n+1-1}{20}+\frac{6(n+1)}{25}+\frac{(n+1+1)}{10}\\
      &=\left(\frac{n+1}{2}\right)\left(\frac{3}{10}+\frac{12}{25}+\frac{1}{5}\right)+\left(\frac{1}{10}-\frac{3}{20}\right)\\
      &=\left(\frac{n+1}{2}\right)\left(\frac{49}{50}\right)-\frac{1}{20}\\
        &<\frac{n+1}{2}.
\end{align*}
That is $\vartheta_n(z)$ satisfies the conditions of \Cref{Corollary1-Theorem4-Omega-FMP}, and hence $\vartheta_n\in\Omega_n$. Indeed, for $z\in\mathbb{D}$, we have
\begin{align*}
\left|z\vartheta_n'(z)-\vartheta_n(z)\right|&=\left|-\frac{3n}{20(n+1)}z^{n+1}+\frac{6(n+1)}{25(n+2)}z^{n+2}+\frac{(n+2)}{10(n+3)}z^{n+3}\right|\\
                                   &<\frac{3n}{20(n+1)}+\frac{6(n+1)}{25(n+2)}+\frac{(n+2)}{10(n+3)}\\
                                     &\leq\frac{3}{20}+\frac{6}{25}+\frac{1}{10}=\frac{49}{100}<\frac{1}{2}.
\end{align*}
\end{example}
We now use \Cref{Theorem2-Suff-Cnd-Omega-FMP} to construct functions involving double integrals that are members of the function class $\Omega_n$.
\begin{theorem}\label{Theorem3-Making-Omega-FMP}
Let $\gamma\geq1,\,n\in\mathbb{N}$, and let $\mathcal{J}(z)$ be analytic in $\mathbb{D}$ such that
\begin{align}\label{Eq:1-Theorem3-Making-Omega-FMP}
\left|\mathcal{J}(z)\right|\leq\frac{n+\gamma}{2}.
\end{align}
Then the function
\begin{align}\label{Eq:2-Theorem3-Making-Omega-FMP}
f(z)=z+z^{n+1}\iint_0^1\mathcal{J}(stz)s^{n-1}t^{n+\gamma-1}dsdt
\end{align}
belongs to the class $\Omega_n$. Moreover, if equality holds in {\rm(\ref{Eq:1-Theorem3-Making-Omega-FMP})}, then the function in {\rm(\ref{Eq:2-Theorem3-Making-Omega-FMP})} becomes $\widehat{f}_{n,\mu}\in\Omega_n$.
\end{theorem}
\begin{proof}
Let us consider the function $f\in\mathcal{A}_n$ satisfying the second-order differential equation
\begin{align}\label{Eq:3-Theorem3-Making-Omega-FMP}
zf''(z)+(\gamma-1)\left(f'-\frac{f(z)}{z}\right)=z^n\mathcal{J}(z).
\end{align}
From (\ref{Eq:1-Theorem3-Making-Omega-FMP}) and (\ref{Eq:3-Theorem3-Making-Omega-FMP}), we have
\begin{align*}
\left|zf''(z)+(\gamma-1)\left(f'-\frac{f(z)}{z}\right)\right|<\frac{n+\gamma}{2}.
\end{align*}
Therefore, in view of \Cref{Theorem2-Suff-Cnd-Omega-FMP}, it follows that the solution of the differential equation (\ref{Eq:3-Theorem3-Making-Omega-FMP}) must lie in the function class $\Omega_n$. We show that the solution of (\ref{Eq:3-Theorem3-Making-Omega-FMP}) is the function given in (\ref{Eq:2-Theorem3-Making-Omega-FMP}). Writing
\begin{align*}
q(z)=f'(z)-\frac{f(z)}{z},
\end{align*}
the equation (\ref{Eq:3-Theorem3-Making-Omega-FMP}) reduces to the form
\begin{align*}
z^{1-\gamma}\big(z^\gamma q(z)\big)'=z^n\mathcal{J}(z).
\end{align*}
This, on solving, gives
\begin{align*}
q(z)=z^n\int_0^1\mathcal{J}(tz)t^{n+\gamma-1}dt,
\end{align*}
or, equivalently
\begin{align}\label{Eq:4-Theorem3-Making-Omega-FMP}
f'(z)-\frac{f(z)}{z}=z^n\int_0^1\mathcal{J}(tz)t^{n+\gamma-1}dt.
\end{align}
Now the differential equation (\ref{Eq:4-Theorem3-Making-Omega-FMP}) can be written as
\begin{align*}
z\left(\frac{f(z)}{z}-1\right)'=z^n\int_0^1\mathcal{J}(tz)t^{n+\gamma-1}dt,
\end{align*}
whose solution is
\begin{align*}
f(z)=z+z^{n+1}\iint_0^1\mathcal{J}(stz)s^{n-1}t^{n+\gamma-1}dsdt.
\end{align*}
If equality holds in (\ref{Eq:1-Theorem3-Making-Omega-FMP}), then we have
\begin{align*}
\mathcal{J}(z)=\mu\frac{n+\gamma}{2} \text{ with } |\mu|=1.
\end{align*}
Substituting this in (\ref{Eq:2-Theorem3-Making-Omega-FMP}), we obtain the function $\widehat{f}_{n,\mu}(z)$.
\end{proof}
\begin{corollary}
Let $\mathcal{J}\in\mathcal{H}$ such that $\left|\mathcal{J}(z)\right|\leq1$. Then the function
\begin{align*}
f(z)=z+z^{2}\iint_0^1\mathcal{J}(stz)tdsdt
\end{align*}
belongs to the class $\Omega$.
\end{corollary}
We conclude this section by showing that for each $n\in\mathbb{N}$ and for each $\mu\in\mathbb{C}$ with $|\mu|=1$, the function $\widehat{f}_{n,\mu}(z)$ is an extreme point of $\Omega_n$. We prove it by showing that $\widehat{f}_{n,\mu}(z)$ satisfies the condition established by Peng and Zhong \cite[Theorem 3.14]{2017-Acta-Math-Omega-Peng-Zhong}. Observe that $\widehat{f}_{n,\mu}(z)$
\begin{align*}
\widehat{f}_{n,\mu}(z)=z+\frac{\mu}{2n}z^{n+1}=z+\frac{1}{2}z\int_0^z\phi(\xi),
\end{align*}
where $\phi(\xi)=\mu\xi^{n-1}$ satisfies
\begin{align*}
\int_0^{2\pi}\log\left[1-|\phi(e^{i\theta})|\right]d\theta %=\int_0^{2\pi}\log[0]d\theta
                                                            =-\infty.
\end{align*}
\section{Subordination Result for A Subclass of $\Omega_n$}\label{Section-Subclass-OmegaCap-FMP}
For $n\in\mathbb{N}$, let us define the set $\widehat{\Omega}_n$ as follows
\begin{align}\label{Eq:Omega-n-Cap-Def-FMP}
\widehat{\Omega}_n:
  =\left\{z+\sum_{k=n+1}^{\infty}a_kz^k\in\mathcal{A}_n:\sum_{k=n+1}^{\infty}k(k-1)|a_k|\leq\frac{n+1}{2}\right\}
\end{align}
Clearly $\widehat{\Omega}_n$ is non-empty, as $\widehat{f}_{n,\mu}\in\widehat{\Omega}_n$, and $\widehat{\Omega}_n\subset\Omega_n$ follows from \Cref{Corollary1-Theorem4-Omega-FMP}. Also, it is obvious from the definition of $\widehat\Omega_n$ and \Cref{Corollary1-Theorem4-Omega-FMP} that $f(z)=z+\sum_{k=n+1}^{\infty}a_kz^k\in\widehat\Omega_n$ if and only if
\begin{align*}
\sum_{k=n+1}^{\infty}k(k-1)|a_k|\leq\frac{n+1}{2}.
\end{align*}
For $n=1$, write $\widehat\Omega=\widehat\Omega_1$ defined as
\begin{align}\label{Eq:Omega-Cap-Defn-FMP}
\widehat{\Omega}:=\left\{z+\sum_{k=2}^{\infty}a_kz^k\in\mathcal{A}:\sum_{k=2}^{\infty}k(k-1)|a_k|\leq1\right\}.
\end{align}
As $\widehat{f}_{1,\mu}(z)\in\widehat\Omega$ for every $\mu$ with $|\mu|=1$, the class $\widehat\Omega$ is not empty. Moreover, $\widehat\Omega$ is a proper subset of $\Omega$, because the function $\varphi(z)=z-z^2/5-z^3/8$ is in $\Omega$ but not in $\widehat\Omega$.
We now prove a subordination theorem for the function class $\widehat\Omega_n$ involving Hadamard product. The following criterion for a sequence of complex numbers $\{s_k\}_{k=1}^\infty$ to be a subordinating factor sequence was established by Wilf \cite{HS-Wilf}.
\begin{lemma}\label{Lemma-Sub-Th-FMP}
The sequence $\{s_k\}_{k=1}^\infty$ is a subordinating factor sequence if and only if
\begin{align*}
\mathrm{Re}\left(1+2\sum_{k=1}^\infty s_kz^k\right)>0.
\end{align*}
\end{lemma}
\begin{theorem}\label{Subordination-Th-FMP}
Let $n\in\mathbb{N}$, and let $f(z)=z+\sum_{k=n+1}^{\infty}a_kz^k\in\widehat\Omega_n$. Then for every convex function $\nu(z)=z+\sum_{k=2}^\infty c_kz^k$ in $\mathbb{D}$, we have
\begin{align}\label{Sub-Th-Bound1-FMP}
\tau(f*\nu)(z)\prec \nu(z)
\end{align}
and
\begin{align}\label{Sub-Th-Bound2-FMP}
\mathrm{Re}\left(f(z)\right)>-\frac{1}{2\tau},
\end{align}
where
\begin{align*}
\tau=\frac{n}{2n+1}.
\end{align*}
The number $\tau$ is the best possible estimate.
\end{theorem}
\begin{proof}
Since $f\in\widehat\Omega_n$, we have
\begin{align}\label{Eq:Omega-Hat-Def-FMP}
\sum_{k=n+1}^{\infty}k(k-1)|a_k|\leq\frac{n+1}{2}.
\end{align}
Also, from the given representations of $f(z)$ and $\nu(z)$, and the definition of Hadamard product, we have
\begin{align*}
\tau(f*\nu)(z)=\tau\left(z+\sum_{k=n+1}^{\infty}a_kc_kz^k\right)=\sum_{k=1}^{\infty}s_kc_kz^k,
\end{align*}
where $c_1=1$ and
\begin{align*}
s_k=
\begin{cases}
    \tau       &\text{ for } k=1\\
     0          &\text{ for } 2\leq k\leq n\\
    \tau{a_k}    & \text{ for } k\geq{n+1}.
\end{cases}
\end{align*}
Clearly, the subordination (\ref{Sub-Th-Bound1-FMP}) will hold if we prove that $\{s_k\}_{k=1}^\infty$ is a subordinating factor sequence. In view of \Cref{Lemma-Sub-Th-FMP}, it is sufficient to prove that
\begin{align*}
\mathrm{Re}\left(1+2\sum_{k=1}^\infty s_kz^k\right)>0.
\end{align*}
Now using (\ref{Eq:Omega-Hat-Def-FMP}) and the fact that the sequence $\{k(k-1)\}_{k=n+1}^\infty$ is an increasing sequence, we have for $|z|=r<1$,
\begin{align*}
\mathrm{Re}\left(1+2\sum_{k=1}^\infty s_kz^k\right)&=\mathrm{Re}\left(1+2\tau z+2\sum_{k=n+1}^\infty\tau a_kz^k\right)\\
                                              &\geq 1-2\tau|z|-2\tau\sum_{k=n+1}^\infty|a_k||z|^k\\
                                              &= 1-2\tau r-2\tau\sum_{k=n+1}^\infty|a_k|r^k\\
                                              %%%%%%&= 1-2\tau r-2\sum_{k=n+1}^\infty\tau|a_k|r\\
                                              %%%%%%&=1-2\tau r-2\tau\frac{1}{n(n+1)}\sum_{k=n+1}^\infty{n(n+1)}|a_k|r\\
                                              &\geq 1-2\tau r-2\tau\frac{1}{n(n+1)}\sum_{k=n+1}^\infty{k(k-1)}|a_k|r\\
                                              &\geq 1-2\tau r-2\tau\frac{1}{n(n+1)}\frac{n+1}{2}r\\
                                              %%%%%&=1-2\tau r-\frac{\tau}{n}r\\
                                              &=1-r\left(2\tau+\frac{\tau}{n}\right)\\
                                              %%%%%&= 1-r\left(\frac{2n}{2n+1}+\frac{1}{2n+1}\right)\\
                                              &=1-r>0.
\end{align*}
Thus (\ref{Sub-Th-Bound1-FMP}) holds true for every function $f\in\widehat\Omega$. If we choose the function $\nu(z)$ as the convex function $\ell(z)=z/(1-z)$ the inequality (\ref{Sub-Th-Bound2-FMP}) follows.
%%Now (\ref{Sub-Th-Bound2-FMP}) follows from \Cref{Note1-FMP} and the fact that $f\prec g$ implies that range of $f$ is a subset of the range of $g$.
A simple observation shows that the function
\begin{align*}
\widehat{f}_{n,-1}(z)=z-\frac{1}{2n}z^{n+1}\in\Omega_n
\end{align*}
which guarantees that the number $\tau$ cannot be replaced by any larger one.
\end{proof}
The following subordination result for $\Omega$ is an immediate consequence of \Cref{Subordination-Th-FMP}.
\begin{corollary}\label{Corollary1-Sub-Th-FMP}
Let $f\in\widehat\Omega$. Then for every convex function $\nu(z)$ in $\mathbb{D}$, we have
\begin{align*}
\frac{1}{3}(f*\nu)(z)\prec \nu(z)
\end{align*}
and
\begin{align}\label{Sub-Th-Coro-Bound2-FMP}
\mathrm{Re}(f(z))>-\frac{3}{2}.
\end{align}
The sharpness of the estimate $\frac{1}{3}$ is guaranteed by the function  $\widehat{f}_{1,-1}(z)=z-z^{2}/2\in\Omega$ \rm{(see  \Cref{Figure_Subordination_Sharpness_1by3_FMP})}.
\end{corollary}
\begin{remark}
The validity of the inequality {\rm(\ref{Sub-Th-Coro-Bound2-FMP})} for functions in $\Omega$ is evident from the growth theorem {\rm\cite[Theorem 3.1]{2017-Acta-Math-Omega-Peng-Zhong}}.
\end{remark}
\begin{figure}[h]
\includegraphics{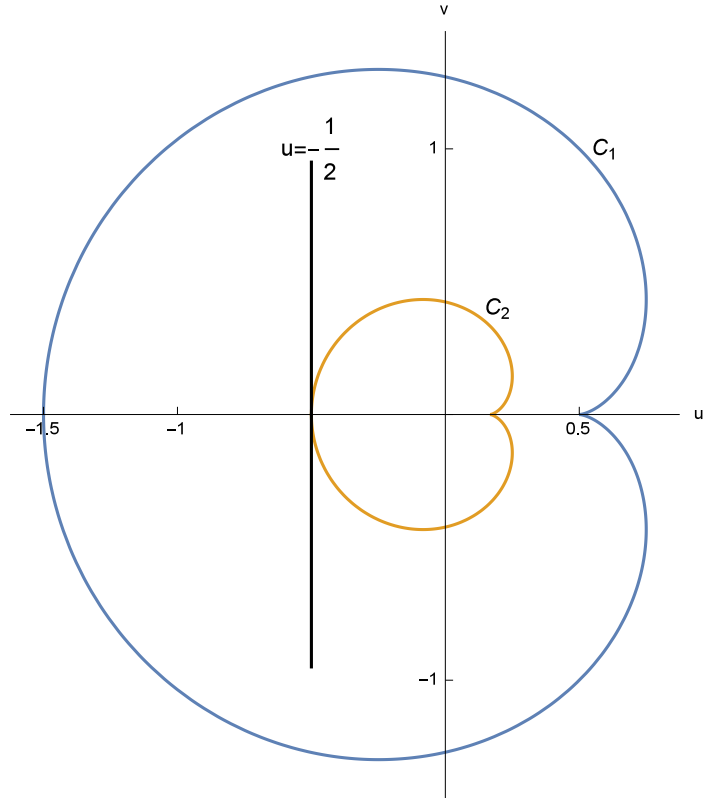}
\caption{$C_1$ is the Boundary curve of $\widehat{f}_{1,-1}(\mathbb{D})$ and $C_2$ is that of $\frac{1}{3}\widehat{f}_{1,-1}(\mathbb{D})$, where $\widehat{f}_{1,-1}(z)=z-z^{2}/2$.}
\label{Figure_Subordination_Sharpness_1by3_FMP}
\end{figure}
\section{Inclusion Properties of $\Omega$}\label{Section-Inclusion-Properties-FMP}
We start this section by stating a necessary condition for a function $f\in\mathcal{A}$ to be in the function class $\Omega$.
\begin{lemma}[{\cite[Corollary 3.12]{2017-Acta-Math-Omega-Peng-Zhong}}]\label{Lemma-Necessary-Cond-Omega-FMP}
If $f(z)=z+\displaystyle\sum_{k=2}^\infty{a_k}z^k$ is in $\Omega$, then
\begin{align}\label{Necessary-Cond-Omega-FMP}
|a_k|\leq\frac{1}{2(k-1)},\quad k\geq2.
\end{align}
\end{lemma}
As mentioned earlier, $\Omega\subset\mathcal{S}^*$ \cite[Theorem 3.1]{2017-Acta-Math-Omega-Peng-Zhong}, the fact that the Koebe function $z/(1-z)^2=z+\sum_{k=2}^\infty{kz^k}\in\mathcal{S}^*$ does not satisfy the necessary condition (\ref{Necessary-Cond-Omega-FMP}) implies that this containment is proper. In this section, we discuss some inclusion type relations existing between $\Omega,\,\mathcal{S}_p$ and $UST$.
%%Having this in mind, it is natural to think for some inclusive relations between $\Omega$ and the other proper subclasses of $\mathcal{S}^*$.
\begin{definition}[{{\tt Parabolic Starlike Functions} ($\mathcal{S}_p$)}\cite{Ronning-1993-UCV-PAMS}]
A function $f\in\mathcal{A}$ is said to be in the class $\mathcal{S}_p\subset\mathcal{S}^*$ if and only if
\begin{align*}
\mathrm{Re}\left(\frac{zf'(z)}{f(z)}\right)>\left|\frac{zf'(z)}{f(z)}-1\right|, \quad z\in\mathbb{D}.
\end{align*}
\end{definition}
These functions were introduced by Ronning \cite{Ronning-1993-UCV-PAMS} and later studied and generalized by many authors (see \cite{Bharati-Swami-1997-Tamkang,Kanas-Wisni-1999-Conic-Regions-JCAM,Kanas-Wisni-2000-Conic-Domains-Starlike-Roman}). Geometrically, $f\in\mathcal{S}_p$ if and only if all the values taken by the expression $zf'(z)/f(z)$ lie in the parabolic region
\begin{align*}
{\bf R_p}:=\left\{w=u+iv\in\mathbb{C}:v^2<2u-1\right\}.
\end{align*}
\begin{lemma}[{\cite[Theorem 3]{Ronning-1993-UCV-PAMS}}]\label{Lemma-Ronning-Parabolic-Star-FMP}
The function $f_k(z)=z+a_kz^k$ is in $\mathcal{S}_p$ if and only if
\begin{align*}
|a_k|\leq\frac{1}{(2k-1)},\quad k\geq2.
\end{align*}
\end{lemma}
\begin{lemma}[{\cite[Corollary 2.4]{Kanas-Wisni-2000-Conic-Domains-Starlike-Roman}}]\label{Lemma-Kanas-Parabolic-Star-FMP}
Let $f(z)=z+\displaystyle\sum_{k=2}^{\infty}a_kz^k\in\mathcal{A}$. If
\begin{align*}
\sum_{k=2}^{\infty}(2k-1)|a_k|\leq1,
\end{align*}
then $f\in\mathcal{S}_p$.
\end{lemma}
\begin{lemma}\label{Lemma-iff-Omega-FMP}
The function $f_k(z)=z+a_kz^k$ is in $\Omega$ if and only if
\begin{align}\label{Eq:Lemma-iff-Omega-FMP}
|a_k|\leq\frac{1}{2(k-1)}, \quad k\geq2.
\end{align}
\end{lemma}
\begin{proof}
The necessary part easily follows from \Cref{Lemma-Necessary-Cond-Omega-FMP}. We now suppose (\ref{Eq:Lemma-iff-Omega-FMP}) holds, then
\begin{align*}
\left|zf_k'(z)-f_k(z)\right|=\left|(k-1)a_kz^k\right|<(k-1)|a_k|\leq\frac{1}{2}.
\end{align*}
Hence, $f_k\in\Omega$.
\end{proof}
If we consider the function $\widehat{f}_{1,\mu}(z)=z+\frac{\mu}{2}z^2\in\Omega$, then it easily follows from \Cref{Lemma-Ronning-Parabolic-Star-FMP} that $\Omega\not\subset\mathcal{S}_p$. For the other direction, we give the following result.
\begin{theorem}\label{Th-Sp-Implies-Omega-FMP}
If $f_k(z)=z+a_kz^k$ belongs to $\mathcal{S}_p$, then $f_k\in\Omega$ for every $k\geq2$.
\end{theorem}
\begin{proof}
As $f_k(z)=z+a_kz^k\in\mathcal{S}_p$, \Cref{Lemma-Ronning-Parabolic-Star-FMP} gives that
\begin{align*}
|a_k|\leq\frac{1}{(2k-1)},\quad k\geq2.
\end{align*}
Since,
\begin{align*}
\frac{1}{(2k-1)}<\frac{1}{2(k-1)}\,\text{ for all } k\geq2,
\end{align*}
the desired result follows from \Cref{Lemma-iff-Omega-FMP}.
\end{proof}
The following result holds for the functions in $\widehat{\Omega}\subset\Omega$ with certain restrictions.
\begin{theorem}
Let $f(z)=z+\displaystyle\sum_{k=3}^{\infty}a_kz^k\in\widehat\Omega$ (i.e., $a_2=0$). Then $f\in\mathcal{S}_p$.
\end{theorem}
\begin{proof}
$f\in\widehat\Omega$ gives
\begin{align}\label{Eq:1-Uniform-ST-FMP}
\sum_{k=2}^{\infty}k(k-1)|a_k|\leq 1.
\end{align}
Making use of (\ref{Eq:1-Uniform-ST-FMP}), we have
\begin{align*}
\sum_{k=2}^{\infty}(2k-1)|a_k|=\sum_{k=3}^{\infty}(2k-1)|a_k|\leq\sum_{k=3}^{\infty}k(k-1)|a_k|
                                  =\sum_{k=2}^{\infty}k(k-1)|a_k|\leq 1.
\end{align*}
Therefore, it follows from \Cref{Lemma-Kanas-Parabolic-Star-FMP} that $f\in\mathcal{S}_p$.
\end{proof}
\begin{definition}[{{\tt Uniformly Starlike Functions} ($UST$) \cite{Goodman-1991-UST-JMAA}}]
A function $f\in\mathcal{A}$ is said to be in the class $UST\subset\mathcal{S}^*$ if and only if
\begin{align*}
\mathrm{Re}\left(\frac{(z-\xi)f'(z)}{f(z)-f(\xi)}\right)\geq 0
\end{align*}
for every pair $(\xi,z)\in\mathbb{D}\times\mathbb{D}$.
\end{definition}
This class of functions was introduced by Goodman \cite{Goodman-1991-UST-JMAA}. These functions have the property that for every circular arc $\gamma$ contained in $\mathbb{D}$, with center $\zeta$ also in $\mathbb{D}$, the arc $f(\gamma)$ is starlike with respect to $f(\zeta)$.
\begin{lemma}[{\cite[Theorem 5]{Merkes-Salmassi-1992-UST}}]\label{Lemma-UST_Merkes-FMP}
If
\begin{align*}
|a_k|\leq\sqrt{\frac{k+1}{2k^3}}, \quad k\geq2,
\end{align*}
then the function $f_k(z)=z+a_kz^k$ is in $UST$.
\end{lemma}
Using \Cref{Lemma-UST_Merkes-FMP}, we prove the following theorem.
\begin{theorem}\label{Th-Omega-Implies-UST-FMP}
Let $f_k(z)=z+a_kz^k$ be in $\Omega$. Then, for all $k\geq3$, $f(z)$ is in $UST$.
\end{theorem}
\begin{proof}
Given $f_k(z)=z+a_kz^k\in\Omega$, we have from \Cref{Lemma-iff-Omega-FMP} that
\begin{align*}
|a_k|\leq\frac{1}{2(k-1)}, \quad k\geq2.
\end{align*}
Since,
\begin{align*}
\frac{1}{2(k-1)}\leq\sqrt{\frac{k+1}{2k^3}}\,\text{ for all } k\geq3,
\end{align*}
it follows from \Cref{Lemma-UST_Merkes-FMP} that $f_k(z)=z+a_kz^k$ is in $UST$ for all $k\geq3$.
\end{proof}
We note that $\mathcal{S}_p\not\subset{UST}$ and $UST\not\subset\mathcal{S}_p$ (cf. \cite[page 21]{Ali-Ravi-2011-UC-UST-Ramamujan}). In view of \Cref{Th-Sp-Implies-Omega-FMP} and \Cref{Th-Omega-Implies-UST-FMP}, we remark the following important result which is not available in the literature. This result gives a kind of inclusion relation between $\mathcal{S}_p$ and $UST$.
\begin{remark}
If $f_k(z)=z+a_kz^k$ is in $\mathcal{S}_p$, then $f_k\in{UST}$ for all $k\geq3$.
\end{remark}
\section{Radii Problems for $\Omega$}\label{Section-Radii-Problems-FMP}
By a radius problem, we mean the following: For two families $\mathcal{F}_1,\mathcal{F}_2$ in $\mathcal{A}$, we say that the number $\rho\;(0<\rho\leq1)$ is the $\mathcal{F}_1$-radius for $\mathcal{F}_2$, if $\rho$ is the largest number such that for every $r$ satisfying $0<r\leq\rho$ we have
\begin{align*}
\frac{1}{r}f(rz)\in\mathcal{F}_1 \quad \text{ for all } f\in\mathcal{F}_2.
\end{align*}
We note that it has been proved \cite[Theorem 3.4]{2017-Acta-Math-Omega-Peng-Zhong} that the $\mathcal{C}$-radius for $\Omega$ is $\frac{1}{2}$. In this section, we will prove some more radii results for the class $\Omega$. Before proceeding, we list out some lemmas that are useful for our discussion.
\begin{lemma}[{\cite[Theorem 3.1]{2017-Acta-Math-Omega-Peng-Zhong}}]\label{Growth-Distort-Theorem-Peng-Z-FMP}
If $f\in\Omega$, then
\begin{align}\label{Eq:Growth-Theorem-FMP}
|z|-\frac{1}{2}|z|^2\leq&|f(z)|\leq |z|+\frac{1}{2}|z|^2,
\end{align}
and
\begin{align}\label{Eq:Distortion-Theorem-FMP}
1-|z|\leq|f'(z)|\leq 1+|z|.
\end{align}
Further, for each $0\neq{z}\in\mathbb{D}$, equality occurs in both the estimates if\, and only if\,
\begin{align}\label{Function-with-Equality-GD-FMP}
f(z)=\widehat{f}_{1,\mu}(z)=z+\frac{\mu}{2}z^2 \text{ with }|\mu|=1.
\end{align}
\end{lemma}
\begin{lemma}\label{Main-Lemma-for-radius-results-FMP}
Let $f\in\Omega$. Then for $|z|=r<1$, we have the sharp estimate
\begin{align*}
\left|\frac{zf'(z)}{f(z)}-1\right|\leq\frac{r}{2-r}.
\end{align*}
\end{lemma}
\begin{proof}
Since $f\in\Omega$, we have
\begin{align*}
\left|zf'(z)-f(z)\right|<\frac{1}{2},
\end{align*}
This can be equivalently written in the equation form as
\begin{align*}
zf'(z)-f(z)=\frac{1}{2}z^2\phi(z),
\end{align*}
where $\phi(z)\in\mathcal{H}$ and $|\phi(z)|\leq1$.
This further implies that
\begin{align}\label{Equivalent-def-members-Omega-FMP}
\left|zf'(z)-f(z)\right|\leq\frac{1}{2}|z|^2.
\end{align}
Inequality (\ref{Equivalent-def-members-Omega-FMP}) along with (\ref{Eq:Growth-Theorem-FMP}) yields
\begin{align*}
\left|\frac{zf'(z)}{f(z)}-1\right|=\frac{1}{|f(z)|}\left|zf'(z)-f(z)\right|
                                     \leq \frac{\frac{1}{2}|z|^2}{|z|-\frac{1}{2}|z|^2}=\frac{r}{2-r}.
\end{align*}
The sharpness of the estimate follows from \Cref{Growth-Distort-Theorem-Peng-Z-FMP}.
\end{proof}
\subsection*{$\mathcal{S}^*(\alpha)$-radius of $\Omega$}
\begin{theorem}
The $\mathcal{S}^*(\alpha)$-radius for the class $\Omega$ is $r(\alpha)=\frac{2(1-\alpha)}{2-\alpha}$, where $0\leq\alpha<1$.
\end{theorem}
\begin{proof}
It is easy to see that
\begin{align*}
\left|\frac{zf'(z)}{f(z)}-1\right|\leq 1-\alpha \implies \mathrm{Re}\frac{zf'(z)}{f(z)}>\alpha.
\end{align*}
If $f\in\Omega$ and $|z|=r$, then we conclude from \Cref{Main-Lemma-for-radius-results-FMP} that
\begin{align*}
\left|\frac{zf'(z)}{f(z)}-1\right|\leq 1-\alpha ~~\text{ if } ~~\frac{r}{2-r}\leq 1-\alpha.
\end{align*}
The later inequality holds true if $r\leq2(1-\alpha)/(2-\alpha)=r(\alpha)$. For sharpness, we consider the function $\widehat{f}_1(z)=z+z^2/2\in\Omega$. At the point $z_0=-2(1-\alpha)/(2-\alpha)$ lying on the circle $|z|=r(\alpha)$, we have
\begin{align*}
\frac{z_0\widehat{f}_1'(z_0)}{\widehat{f}_1(z_0)}=\frac{1+z_0}{1+z_0/2}=\alpha.
\end{align*}
%%Or, $\mathrm{Re}\left(z_0\widehat{f}_1'(z_0)/\widehat{f}_1(z_0)\right)=\alpha$.
This proves that the result is sharp.
\end{proof}
\begin{remark}
The $\mathcal{S}^*\left(\frac{1}{2}\right)$-radius for the class $\Omega$ is $\frac{2}{3}$.
\end{remark}
\subsection*{$\mathcal{S}_p$-radius of $\Omega$}
\begin{theorem}
The $\mathcal{S}_p$-radius for the class $\Omega$ is $\frac{2}{3}$.
\end{theorem}
\begin{figure}[h]
\includegraphics{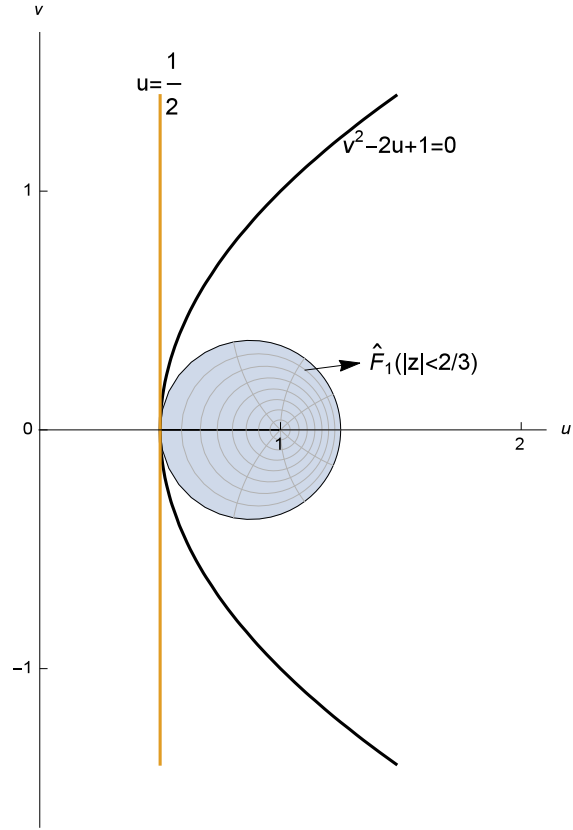}
\caption{Sharpness of $\mathcal{S}_p$-radius: $\widehat{F}_1(z)=z\widehat{f}'_1(z)/\widehat{f}_1(z)$ with $\widehat{f}_1(z)=z+z^2/2$.}
\label{Figure-Sharpness-Sp-Radius-2/3-FMP}
\end{figure}
\begin{proof}
Let $f\in\Omega$. Then by a well known result \cite[p. 21]{Ali-Ravi-2011-UC-UST-Ramamujan} we get that
\begin{align*}
\left|\frac{zf'(z)}{f(z)}-1\right|<\frac{1}{2} \implies f\in\mathcal{S}_p.
\end{align*}
In view of \Cref{Main-Lemma-for-radius-results-FMP}, this inequality is true if $r/(2-r)<1/2$, that is if $r<2/3$. Further, we note that $f\in\mathcal{S}_p$ implies (see \cite{Ronning-1993-UCV-PAMS})
\begin{align*}
\mathrm{Re}\left(\frac{zf'(z)}{f(z)}\right)>\frac{1}{2}.
\end{align*}
We now show that there exists at least one function $f\in\Omega$ for which $\mathrm{Re}\left(zf'(z)/f(z)\right)=1/2$ at some point on the circle $|z|=2/3$. This can be easily shown if we take $f(z)=\widehat{f}_1(z)=z+z^2/2\in\Omega$ and $z=-2/3$. Hence the result is sharp (see Figure \ref{Figure-Sharpness-Sp-Radius-2/3-FMP}).
\end{proof}
\begin{remark}
For $\Omega$, we have $\mathcal{S}^*\left(\frac{1}{2}\right)-radius=\mathcal{S}_p-radius=\frac{2}{3}$.
\end{remark}

In 1992, Ma and Minda \cite{Ma-Minda-1992-A-unified-treatment} introduced a general method of constructing function classes $\mathcal{S}^*(\varphi)\subset\mathcal{S}^*$ as
\begin{align*}
\mathcal{S}^*(\varphi):=\left\{f\in\mathcal{A}:\frac{zf'(z)}{f(z)}\prec\varphi(z)\right\},
\end{align*}
where the function $\varphi:\mathbb{D}\to\mathbb{C}$ satisfies (i) $\varphi(z)$ is univalent with positive real part, (ii) $\varphi(z)$ maps $\mathbb{D}$ onto a region that is starlike with respect to $\varphi(0)=1$, (iii) $\varphi(\mathbb{D})$ is symmetric about the real axis and (iv) $\varphi'(0)>0$. In the recent past, using this approach a number of starlike classes have been introduced and studied along with their geometric properties. Here, we first mention a few of them and then solve the corresponding radius problem for the class $\Omega$.
\subsection*{$\mathcal{S}^*_e$-radius of $\Omega$}
In 2015, the class $\mathcal{S}^*_e=\mathcal{S}^*(e^z)$ was introduced by Mendiratta et al. \cite{Mendiratta-Ravi-2015-Expo-BMMS}. Thus, $\mathcal{S}^*_e$ is the collection of all functions $f\in\mathcal{A}$ satisfying $zf'(z)/f(z)\prec {e}^z$. Equivalently, the function $f\in\mathcal{S}^*_e$ if and only if $zf'(z)/f(z)$ lies in the region
\begin{align*}
{\bf R_e}:=\left\{w\in\mathbb{C}:|\log{w}|<1\right\}.
\end{align*}
Further, if $f\in\mathcal{S}^*_e$, then
\begin{align*}
\frac{1}{e}<\mathrm{Re}\left(\frac{zf'(z)}{f(z)}\right)<e.
\end{align*}
\begin{lemma}[{\cite[Lemma 2.2]{Mendiratta-Ravi-2015-Expo-BMMS}}]\label{Lemma-Mendiratta-Ravi-2015-Expo}
For $1/e<a<e$, let $r_a$ be given by
\begin{align*}
r_a=
\begin{cases}
a-\frac{1}{e}, & \frac{1}{e}<a\leq\frac{1}{2}(e+\frac{1}{e})\\
e-a,            & \frac{1}{2}(e+\frac{1}{e})\leq{a}<e.
\end{cases}
\end{align*}
Then $\left\{w\in\mathbb{C}:|w-a|<r_a\right\}\subset{\bf R_e}$.
\end{lemma}
\begin{theorem}
The $\mathcal{S}_e^*$-radius for the class $\Omega$ is ${\bf{r_e}}=\frac{2\left(1-e^{-1}\right)}{2-e^{-1}}\approx0.77460032643$.
\end{theorem}
\begin{proof}
Let $f\in\Omega$. Then, for $|z|=r<1$, \Cref{Main-Lemma-for-radius-results-FMP} gives
\begin{align}\label{Pengs-Main-Disk-Omega-FMP}
\left|\frac{zf'(z)}{f(z)}-1\right|\leq\frac{r}{2-r},
\end{align}
which is a disk centered at $(1,0)$ and radius $r/(1-r)$. It follows from \Cref{Lemma-Mendiratta-Ravi-2015-Expo} that this disk will be contained in ${\bf R_e}$ if and only if
\begin{align*}
\frac{r}{2-r}<1-\frac{1}{e}.
\end{align*}
This is true if and only if
\begin{align*}
r<\frac{2 \left(1-\frac{1}{e}\right)}{2-\frac{1}{e}}={\bf{r_e}}.
\end{align*}
If we consider the function $\widehat{f}_1(z)=z+z^2/2\in\Omega$, then it is easy to see that at the point $z=-{\bf{r_e}}$, we have
\begin{align*}
\frac{z\widehat{f}_1'(z)}{\widehat{f}_1(z)}=\frac{1}{e}.
\end{align*}
This proves that the result is sharp (see \Cref{Figure-Sharpness-SE-Radius-0.7746-FMP}).
\end{proof}
\begin{figure}[h]
\includegraphics{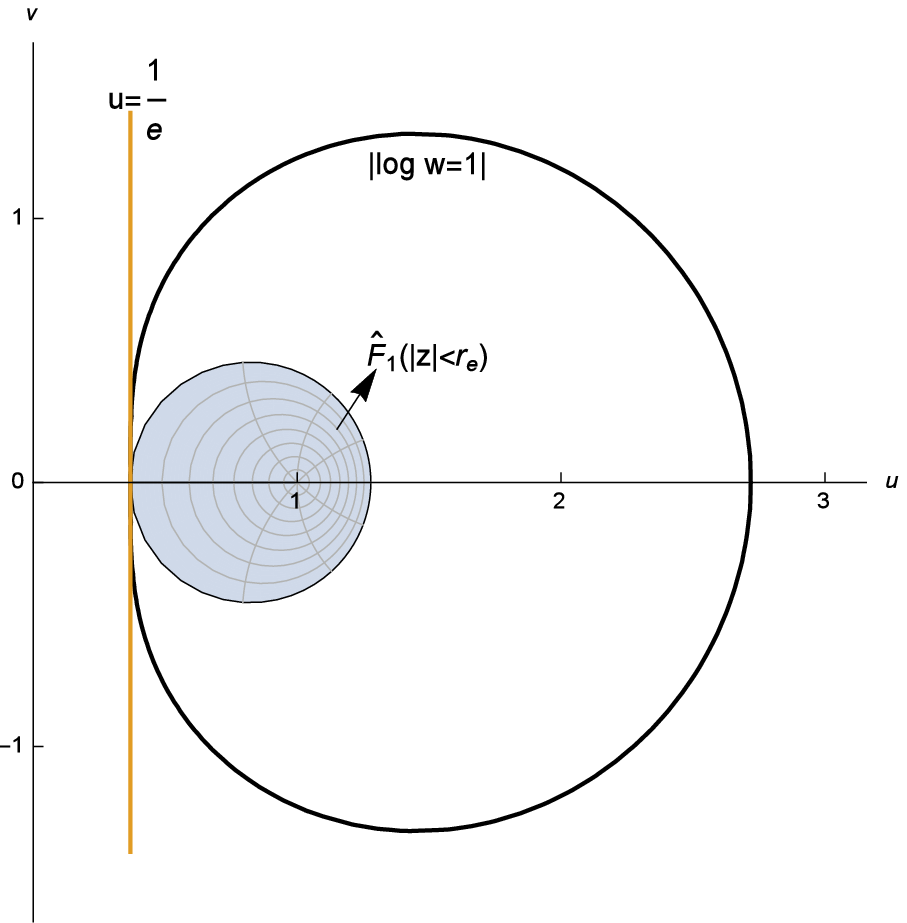}
\caption{Sharpness of $\mathcal{S}_e^*$-radius: $\widehat{F}_1(z)=z\widehat{f}'_1(z)/\widehat{f}_1(z),\;\widehat{f}_1(z)=z+z^2/2$.}
\label{Figure-Sharpness-SE-Radius-0.7746-FMP}
\end{figure}
\subsection*{$\mathcal{S}^*_C$-radius of $\Omega$}
In 2016, Sharma et al. \cite{Sharma-Ravi-2016-Cardioid} introduced and discussed the class $\mathcal{S}^*_C\subset\mathcal{S}^*$ defined as
\begin{align*}
\mathcal{S}^*_C:=\left\{f\in\mathcal{A}:\frac{zf'(z)}{f(z)}\prec 1+\frac{4}{3}z+\frac{2}{3}z^2\right\}.
\end{align*}
A function $f\in\mathcal{A}$ is in the class $\mathcal{S}^*_C$ if and only if ${zf'(z)}/f(z)$ lies in the region bounded by the cardiod
$\left(9 u^2-18 u+9 v^2+5\right)^2-16 \left(9 u^2-6 u+9 v^2+1\right)=0$. We let
\begin{align*}
{\bf R_C}:=\left\{u+iv:\left(9 u^2-18 u+9 v^2+5\right)^2-16 \left(9 u^2-6 u+9 v^2+1\right)<0\right\}.
\end{align*}
\begin{lemma}[{\cite[Lemma 2.5]{Sharma-Ravi-2016-Cardioid}}]\label{Lemma-Sharma-Ravi-2016-Cardioid}
For $1/3<a<3$, let $r_a$ be given by
\begin{align*}
r_a=
\begin{cases}
\frac{3a-1}{3}, & \frac{1}{3}<a\leq\frac{5}{3}\\
3-a,            & \frac{5}{3}\leq{a}<3.
\end{cases}
\end{align*}
Then $\left\{w\in\mathbb{C}:|w-a|<r_a\right\}\subset{\bf R_C}$.
\end{lemma}
\begin{theorem}\label{Theorem-radius-Cardiod-FMP}
The $\mathcal{S}_C^*$-radius for the class $\Omega$ is $\frac{4}{5}$.
\end{theorem}
\begin{proof}
In view of \Cref{Main-Lemma-for-radius-results-FMP} and \Cref{Lemma-Sharma-Ravi-2016-Cardioid}, it follows that for any $f\in\Omega$, the disk (\ref{Pengs-Main-Disk-Omega-FMP}) will lie inside the region ${\bf R_C}$ if and only if
\begin{align*}
\frac{r}{2-r}<\frac{2}{3} \iff r<\frac{4}{5}.
\end{align*}
Again, from \Cref{Lemma-Sharma-Ravi-2016-Cardioid}, it can be easily seen that the largest disk with center at $(1,0)$ and lying completely inside ${\bf R_C}$ is
\begin{align*}
\left\{w:|w-1|<\frac{2}{3}\right\}.
\end{align*}
Clearly the left diametric end point of this disk is $1/3$. The sharpness of our result will follow if we can find at least one  function $f\in\Omega$ and a point on the circle $|z|=4/5$, say $z_0$, such that the value of $zf'(z)/f(z)$ at $z_0$ is $1/3$. We see that one such function in $\Omega$ is $\widehat{f}_1(z)=z+z^2/2$, and the corresponding point on the circle $|z|=4/5$ is $z_0=-4/5$ (see \Cref{Figure-Sharpness-SC-Radius-4/5-FMP}).
\end{proof}
\begin{minipage}{0.45\textwidth}
\centering
 \includegraphics[width=\linewidth]{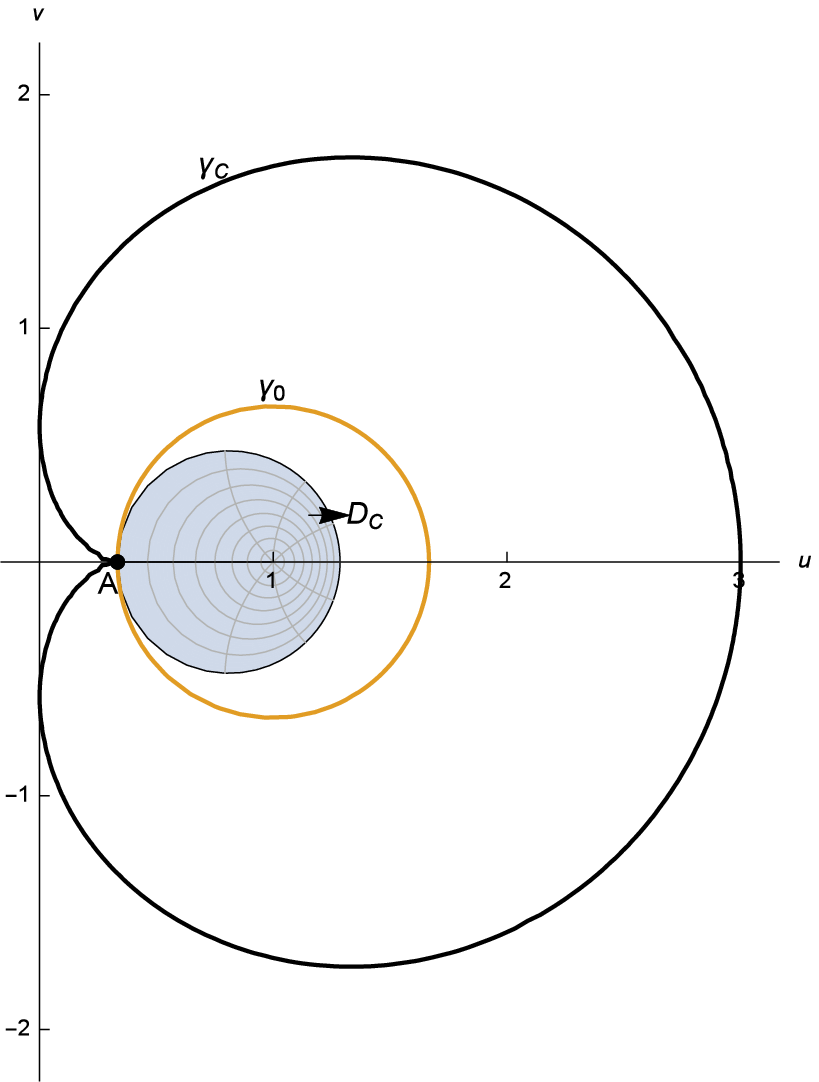}
 %%\captionof{figure}{Sharpness of $\mathcal{S}_C^*$-radius.}
 %%%\label{Figure-Sharpness-SC-Radius-4/5-FMP}
\end{minipage}
\begin{minipage}{0.53\textwidth}
\begin{multline*}
\gamma_C:\left(9 u^2-18 u+9 v^2+5\right)^2\\
                  -16 \left(9 u^2-6 u+9 v^2+1\right)=0.
\end{multline*}
$\gamma_0:\left|w-1\right|=\frac{2}{3}$.\\
\\
$D_C:\widehat{F}_{1}\left(\left|z\right|<\frac{4}{5}\right)$, where $\widehat{F}_1(z)=\frac{z\widehat{f}'_1(z)}{\widehat{f}_1(z)}$\\
with $\widehat{f}_1(z)=z+\frac{z^2}{2}$.\\
 \\
$A=1/3$.
\end{minipage}
\captionof{figure}{Sharpness of $\mathcal{S}_C^*$-radius.}
 \label{Figure-Sharpness-SC-Radius-4/5-FMP}
\subsection*{$\mathcal{S}^*_S$-radius of $\Omega$}
In 2019, using the same Ma-Minda's construction \cite{Ma-Minda-1992-A-unified-treatment}, Cho et al. \cite{Cho-2019-Sine-BIMS} introduced another important class of starlike functions $\mathcal{S}^*_S$ defined as
\begin{align*}
\mathcal{S}^*_S:=\left\{f\in\mathcal{A}:\frac{zf'(z)}{f(z)}\prec 1+\sin{z}\right\}.
\end{align*}
Let us set ${\bf R_S}:=\varphi_{\scriptscriptstyle{S}}(\mathbb{D})$, where $\varphi_{\scriptscriptstyle{S}}(z)=1+\sin{z}$.
\begin{lemma}[{\cite[Lemma 3.3]{Cho-2019-Sine-BIMS}}]\label{Lemma-Cho-2019-Sine-BIMS}
For $1-\sin{1}\leq{a}\leq1+\sin{1}$, let $r_a$ be given by
\begin{align*}
r_a=\sin{1}-|a-1|.
\end{align*}
Then $\left\{w\in\mathbb{C}:|w-a|<r_a\right\}\subset{\bf R_S}$.
\end{lemma}
Using \Cref{Lemma-Cho-2019-Sine-BIMS}, the following theorem can be proved.
\begin{theorem}
The $\mathcal{S}_S^*$-radius for the class $\Omega$ is ${\bf r_S}=\frac{2\sin{1}}{1+\sin{1}}\approx{0.91391174962}$
\rm{(see \Cref{Figure-Sharpness-SS-Radius-0.91391174962-FMP})}.
\end{theorem}
\begin{minipage}{0.55\textwidth}
 \centering
 \includegraphics{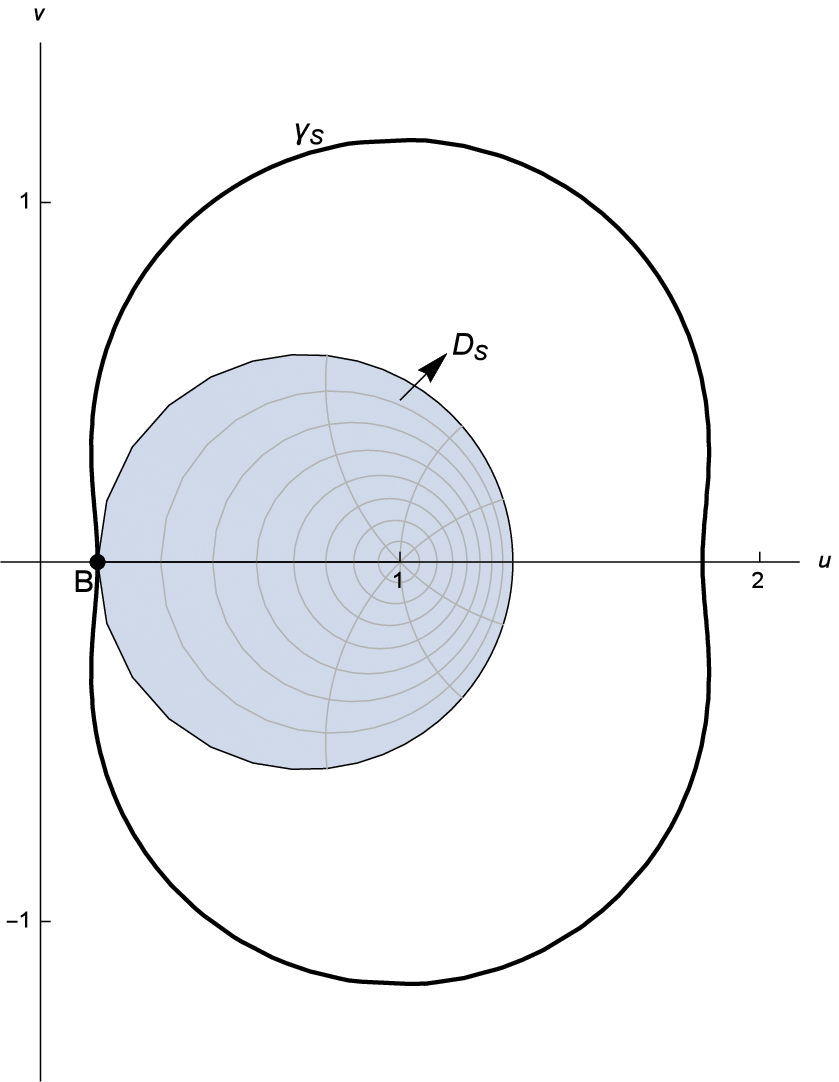}
 \captionof{figure}{Sharpness of $\mathcal{S}_S^*$-radius.}
 \label{Figure-Sharpness-SS-Radius-0.91391174962-FMP}
\end{minipage}
\begin{minipage}{0.45\textwidth}
$\gamma_S:$ Boundary curve of $\varphi_{\scriptscriptstyle{S}}(\mathbb{D})$,\\
where $\varphi_{\scriptscriptstyle{S}}(z)=1+\sin{z}$.\\
\\
$D_S:\widehat{F}_{1}\left(|z|<{\bf r_S}\right),\,\widehat{F}_1(z)=\frac{z\widehat{f}'_1(z)}{\widehat{f}_1(z)}$\\
with $\widehat{f}_1(z)=z+\frac{z^2}{2}$.\\
\\
$B=1-\sin{1}$.
\end{minipage}
\noindent
\subsection*{$\mathcal{S}^*_{SG}$-radius of $\Omega$}
Latestly, Goel and Kumar \cite{Goel-Siva-2019-Sigmoid-BMMS} introduced the starlike class $\mathcal{S}^*_{SG}=\mathcal{S}^*(\varphi)$, where $\varphi(z)=2/(1+e^{-z})$ is the modified sigmoid function. The following radius result can be easily verified by applying
 \cite[Lemma 2.2]{Goel-Siva-2019-Sigmoid-BMMS}.
\begin{theorem}
The $\mathcal{S}_{SG}^*$-radius for the class $\Omega$ is ${\bf r_{SG}}=\frac{e-1}{e}\approx{0.6321205588285577}$.
\end{theorem}
\section{Concluding Remarks and Some Open Problems}\label{Section-Open-Problems-FMP}
{\it Since the members of\, $\Omega$ are starlike, it easily follows from the Alexander's theorem that the members of the class
\begin{align*}
\Upsilon:=\left\{f\in\mathcal{A}:\left|z^2f''(z)\right|<\frac{1}{2}\right\}.
\end{align*}
are convex.}\\
The following result of Kanas and Wisniowska \cite[Corollary 3.2]{Kanas-Wisni-1999-Conic-Regions-JCAM} follows immediately.
\begin{corollary}\label{Corollary-Kanas-Wisni-1999-Conic-Regions-FMP}
If $f(z)=z+\sum_{k=2}^{\infty}a_kz^k\in\mathcal{A}$ satisfies
\begin{align}\label{Kanas-Winsi-Cnvxty-Cond-FMP}
\sum_{k=2}^{\infty}k(k-1)|a_k|\leq\frac{1}{2},
\end{align}
then $f(z)$ is convex.
\end{corollary}
\begin{remark}\label{Remark-Upsilon-Cap-FMP}
If $\widehat{\Upsilon}$ denotes all functions $f\in\mathcal{A}$ satisfying {\rm(\ref{Kanas-Winsi-Cnvxty-Cond-FMP})}, then we have the inclusion $\widehat\Upsilon\subset\mathcal{C}$. Also, from the definition {\rm(\ref{Eq:Omega-Cap-Defn-FMP})}, it follows that $\widehat\Upsilon\subset\widehat\Omega\subset\Omega$.
\end{remark}
The fact that $\Omega$ is closed with respect to Hadamard product \cite[Theorem 3.7]{2017-Acta-Math-Omega-Peng-Zhong}, which is true for the convex class $\mathcal{C}$ also \cite[Theorem 8.6, p. 247]{Duren-1983-Book-UFs}, gives the intuition of some interesting relationship between $\Omega$ and $\mathcal{C}$.  Also, we have been successful in finding a class $\widehat\Upsilon$ such that $\widehat\Upsilon\subset\mathcal{C}$ and $\widehat\Upsilon\subset\Omega$ (see, \Cref{Remark-Upsilon-Cap-FMP}), which further strengthens the above intuitional claim of some kind of relationship between the two.
%%\begin{figure}[H]
%% \includegraphics{Figure_Convolution_g1_Cap_FMP}
%% \caption{Convexity of $\widehat{g}_1(\mathbb{D})=\left(\widehat{f}_1*\widehat{f}_1\right)(\mathbb{D})$.}
%% \label{Figure_Convolution_g1_Cap_FMP}
%%\end{figure}
We conjecture the following result regarding the Hadamard product between the members of $\Omega$.
\begin{conjecture}\label{Conjecture0-FMP}
Let $f_1,f_2\in\Omega$. Then $f_1*f_2\in\mathcal{C}$.
\end{conjecture}
Using the sufficiency condition (\ref{Corollary-Kanas-Wisni-1999-Conic-Regions-FMP}), the truthfulness of the above result can be easily verified for the function $\widehat{f}_{1,\mu}(z)=z+\frac{\mu}{2}z^2\in\Omega$ with $|\mu|=1$, which serves as the extremal for many problems in $\Omega$ \cite{2017-Acta-Math-Omega-Peng-Zhong} (see also \cite{2018-BMMS-Omega-Obra-Peng}). %%%%%(the function with maximum starlikeness)
Indeed, for
\begin{align*}
\widehat{g}_{\mu}(z)=\widehat{f}_{1,\mu}(z)*\widehat{f}_{1,\mu}(z)=z+\frac{\mu^2}{4}z^2,
\end{align*}
we have
\begin{align*}
\sum_{k=2}^\infty{k(k-1)}|a_k|=2\left(\frac{1}{4}\right)=\frac{1}{2}.
\end{align*}
Hence $\widehat{g}_{\mu}(z)\in\mathcal{C}$.
\subsection*{Strongly Starlike Functions}
A function $f\in\mathcal{A}$ is said to be strongly starlike of order $\beta\;(0<\beta\leq1)$ if and only if
\begin{align*}
\left|\arg\left(\frac{zf'(z)}{f(z)}\right)\right|<\frac{\beta\pi}{2}, \quad z\in\mathbb{D}.
\end{align*}
We usually denote this class of functions by $\mathcal{SS}^*(\beta)$. Observe that $\mathcal{SS}^*(1)=\mathcal{S}^*$, and for $0<\beta<1$, the class $\mathcal{SS}^*(\beta)$ consists only of bounded starlike functions, and hence in this case the inclusion $\mathcal{SS}^*(\beta)\subset\mathcal{S}^*$ is proper.
\begin{problem}
To find the $\mathcal{SS}^*(\beta)$-radius for the class $\Omega$.
\end{problem}


\begin{thebibliography}{M}
\bibitem{Ali-Ravi-2011-UC-UST-Ramamujan}
    R. M. Ali and V. Ravichandran, Uniformly convex and uniformly starlike functions. Math. Newsletter {\bf 21} (2011), no. 1, 16-30.
\bibitem{Bharati-Swami-1997-Tamkang}
       R. Bharati, R. Parvatham\ and\ A. Swaminathan, On subclasses of uniformly convex functions and corresponding class of starlike functions, Tamkang J. Math. {\bf 28} (1997), no.~1, 17--32.
\bibitem{Bulboaca-2005-Diff-Sub-Book}       
        T. Bulboac\v{a}, {\it Differential Subordinations and Superordinations}, Recent Results, House of Scientific Book Publ., Cluj-Napoca, 2005.   
\bibitem{Cho-2019-Sine-BIMS}
      N. E. Cho, V. Kumar, S. S. Kumar\ and\ V. Ravichandran, Radius problems for starlike functions associated with the sine function, Bull. Iranian Math. Soc. {\bf 45} (2019), no.~1, 213--232.
\bibitem{Duren-1983-Book-UFs}
         P. L. Duren, {\it Univalent functions}, Grundlehren der Mathematischen Wissenschaften, 259,
               Springer-Verlag, New York, 1983.
\bibitem{Goel-Siva-2019-Sigmoid-BMMS}
     P. Goel\ and\ S. S. Kumar, Certain class of starlike functions associated with modified sigmoid function, Bull. Malays. Math. Sci. Soc. (2019). https://doi.org/10.1007/s40840-019-00784-y.
\bibitem{Goodman-1991-UST-JMAA}
          A. W. Goodman, On uniformly starlike functions, J. Math. Anal. Appl. {\bf 155} (1991), no.~2, 364--370.
\bibitem{Kanas-Wisni-1999-Conic-Regions-JCAM}
     S. Kanas\ and\ A. Wisniowska, Conic regions and $k$-uniform convexity, J. Comput. Appl. Math. {\bf 105} (1999), no.~1-2, 327--336.
\bibitem{Kanas-Wisni-2000-Conic-Domains-Starlike-Roman}
         S. Kanas\ and\ A. Wi\'{s}niowska, Conic domains and starlike functions, Rev. Roumaine Math. Pures Appl. {\bf 45} (2000),
            no.~4, 647--657 (2001).
%%%\bibitem{Y-Komato-1990}
%%%         Y. Komatu, On analytic prolongation of a family of operators, Mathematica (Cluj) {\bf 32(55)} (1990), no.~2, 141--145.
\bibitem{Little-Wood-1925-SUB}
        J. E. Littlewood, On Inequalities in the Theory of Functions, Proc. London Math. Soc. (2)
                     {\bf 23} (1925), no.~7, 481--519.
\bibitem{Ma-Minda-1992-A-unified-treatment}
      W. C. Ma\ and\ D. Minda, A unified treatment of some special classes of univalent functions, in {\it Proceedings of the Conference on Complex Analysis (Tianjin, 1992)}, 157--169, Conf. Proc. Lecture Notes Anal., I, Int. Press, Cambridge.
%%%\bibitem{2019-BMMS-Mahzoon-Omega-Preprint}
%%%          H. Mahzoon, and R. Kargar, Further results for a subclass of univalent functions related with differential equation,
%%%           arXiv preprint arXiv:1901.02408 (2019).
\bibitem{Mendiratta-Ravi-2015-Expo-BMMS}
       R. Mendiratta, S. Nagpal\ and\ V. Ravichandran, On a subclass of strongly starlike functions associated with exponential function, Bull. Malays. Math. Sci. Soc. {\bf 38} (2015), no.~1, 365--386.
\bibitem{Merkes-Salmassi-1992-UST}
          E. Merkes\ and\ M. Salmassi, Subclasses of uniformly starlike functions, Internat. J. Math. Math. Sci. {\bf 15} (1992), no.~3, 449--454.
\bibitem{Miller-Mocanu-Book-2000-Diff-Sub}
       S. S. Miller\ and\ P. T. Mocanu, {\it Differential subordinations}, Monographs and Textbooks in Pure and Applied Mathematics, 225, Marcel Dekker, Inc., New York, 2000.
\bibitem{2018-BMMS-Omega-Obra-Peng}
         M. Obradovi\'{c}\ and\ Z. Peng, Some new results for certain classes of univalent functions,
         Bull. Malays. Math. Sci. Soc. {\bf 41} (2018), no.~3, 1623--1628.
\bibitem{2017-Acta-Math-Omega-Peng-Zhong}
          Z. Peng\ and\ G. Zhong, Some properties for certain classes of univalent functions defined by differential inequalities,
          Acta Math. Sci. Ser. B (Engl. Ed.) {\bf 37} (2017), no.~1, 69--78.
%%%\bibitem{Ravi-Ganita-UCV-2002}
%%%           V. Ravichandran, On uniformly convex functions, Ganita {\bf 53} (2002), no.~2, 117--124.
\bibitem{Ronning-1993-UCV-PAMS}
          F. R\o nning, Uniformly convex functions and a corresponding class of starlike functions, Proc. Amer. Math. Soc. {\bf 118} (1993), no.~1, 189--196.
\bibitem{Sharma-Ravi-2016-Cardioid}
          K. Sharma, N. K. Jain\ and\ V. Ravichandran, Starlike functions associated with a cardioid, Afr. Mat. {\bf 27} (2016), no.~5-6, 923--939.
%%%\bibitem{Silverman-UFWNC-1975-PAMS}
%%%           H. Silverman, Univalent functions with negative coefficients, Proc. Amer. Math. Soc. {\bf 51} (1975), 109--116.
\bibitem{HS-Wilf}
          H. S. Wilf, Subordinating factor sequences for convex maps of the unit circle,
            Proc. Amer. Math. Soc. {\bf 12} (1961), 689--693.
\end{thebibliography}
\end{document}